\documentclass[11pt]{amsart}
\usepackage{xcolor,cancel}
\usepackage[normalem]{ulem}
\usepackage{amsmath}

\makeatother

\usepackage{pb-diagram}
\usepackage{graphicx}

\usepackage{stmaryrd}   

\usepackage{tikz}

\usepackage{fullpage}
\usepackage[latin1]{inputenc}
\usepackage{hyperref}   
\usepackage{enumitem}

\numberwithin{equation}{section}
\theoremstyle{definition}
\newtheorem{remark}{Remark}[section]

\newtheorem{definition}[remark]{Definition}

\theoremstyle{plain}
\newtheorem{theorem}[remark]{Theorem}
\newtheorem{lemma}[remark]{Lemma}

\newtheorem*{thmA}{Theorem A}

\newcommand{\comment}[1]{}            
\newcommand{\mt}[1]{{\text{\rm #1}}}  

\newcommand{\R}{\mathbb{R}}           
\newcommand{\N}{\mathbb{N}}           
\newcommand{\Z}{\mathbb{Z}}           

\renewcommand{\Im}{\mathop{\mt{Im}}}

\DeclareMathOperator{\supp}{\mt{supp}}

\renewcommand{\ln}{\operatorname{\mt{ln}}}

\renewcommand{\exp}{\operatorname{\mt{exp}}}
\renewcommand{\max}{\operatorname*{\mt{max}}}
\renewcommand{\min}{\operatorname*{\mt{min}}}
\renewcommand{\lim}{\operatorname*{\mt{lim}}}

\renewcommand{\inf}{\operatorname*{\mt{inf}}}

\DeclareMathOperator{\id}{\mt{id}}             
\DeclareMathOperator{\pr}{\mt{pr}}             

\DeclareMathOperator{\Ric}{\mt{Ric}}

\newcommand{\SecondFF}{\operatorname{\mathit{II}}}    

\DeclareMathOperator{\inj}{\mt{inj}}

\newcommand{\p}{\partial}
\newcommand{\Rm}{\operatorname{Rm}}

\newcommand{\V}{\noindent}
\newcommand{\benu}{\begin{enumerate}}
\newcommand{\eenu}{\end{enumerate}}
\begin{document}
\vspace*{-6mm}
\title{Cheeger-Gromov compactness for manifolds with boundary}

\author{Olaf M\"uller}
\address{Department of Mathematics\\
Humboldt-Universit\"at, Germany}
\email{mullerol@math.hu-berlin.de}
\begin{abstract}

We prove Cheeger-Gromov convergence for a subsequence of a given sequence of manifolds-with-boundary of bounded geometry. The method of the proof is to reduce, via height functions, the problem to the setting of Hamilton's compactness theorem for manifolds without boundary. 
\end{abstract}  
\maketitle

\vspace*{-8mm}

\section{Introduction and statement of the main result}
\subsection{Motivation}
Let $i \mapsto (M_i,g_i,x^0_i)$ be a sequence of pointed connected Riemannian $n$-dimensional
manifolds\footnote{Throughout this article, the term 'manifold' includes manifolds with boundary.} and we ask whether it {\em subconverges} to a smooth pointed smooth Riemannian manifold $(M_\infty, g_\infty, x^0_\infty)$, i.e., whether a subsequence converges in the Cheeger-Gromov sense to
\begin{equation*}
\displaystyle (M_{\infty},g_{\infty},x^0_{\infty})=  \lim_{i\to \infty}(M_i,g_i,x^0_i)
\end{equation*}

($M_\infty$ might be noncompact even if all $M_i$ are
compact and have non-empty non-compact boundary). This is essentially the question which subsets of the set of Riemannian manifolds are precompact, a question that leads to statements like finiteness of the set of diffeomorphism classes of manifolds with bounds on sectional curvature, diameter and volume, in Cheeger's famous result \cite{jC} that set the ground for this area of geometry. Subsequently, the use of harmonic coordinates (whose existence in uniformly large balls had been shown by Jost and Karcher \cite{JK} under the assumptions of bounds on the sectional curvature) replaced the one of geodesic coordinates, and allowed Peters \cite{sP} and Greene-Wu \cite{GW} to obtain results of better regularity. An interesting recent result in this context is the one by Portegies \cite{jP}, who reproduced Jost's and Karcher's result under the mere assumption of Ricci bounds. For nice accounts on
parts of the subject see \cite{rP},\cite{cS}.

Most results in this context either fix the topology a priori (like the interesting approach taken by Allen and Perales in \cite{AP}) or conclude
Gromov-Hausdorff subconvergence to a metric space (like in Wong's results \cite{W1}, \cite{W2}, which allow concluding subconvergence to an Alexandrov space with convex boundary) or subconvergence in a
space of currents instead of $C^k$-convergence to a Riemannian
manifold as above. An important result \cite{CC}, due to Cheeger and
Colding, shows that the subset of regular points of such a limit, that
is, of points around which the space looks like a manifold, has
full $n$-Hausdorff measure. There are some results that conclude
convergence of a subsequence to a Riemannian manifold: Gromov
\cite{mG} assumes compactness, uniformly bounded diameter and
nonnegative Ricci curvature resp. uniformly bounded sectional
curvature and $1/{\rm vol}$ and concludes Lipschitz resp. metric
convergence to a $C^{1, \alpha}$ resp. $C^0$ metric; Anderson
\cite{mA} shows $C^{1,\alpha}$-precompactness of the class of compact
connected Riemannian $n$-manifolds of bounded Ricci curvature and
diameter and injectivity radius bounded from $0$; the result of
Hamilton \cite{rsH}, enlisted in this article as Theorem \ref{Bamler-thm}, for (compact or non-compact) manifolds without boundary states that under appropriate bounds there is $C^k$-precompactness in the class of open manifolds as well, in a sense specified below. 

Focusing further on $C^k$-convergence of manifolds with nonempty boundary, we find this issue addressed only by the works of Kodani \cite{sK}, Knox
\cite{kK} and Anderson-Katsuda-Kurylev-Lassas-Taylor \cite{AKKLT}. The main results of these three articles assume compactness of the manifolds and a
uniform diameter bound and conclude subconvergence to a Riemannian manifold of at most $C^{1,\alpha}$ regularity. In comparison, Theorem A of the present article assumes, for any desired $k \in \N$, a $C^{k+1}$ bound on sectional curvature\footnote{It is easy to construct counterexamples showing that e.g. without a $C^1$-bound on curvature one cannot obtain $C^3$-subconvergence: There are $g_n := f_n \cdot g$ with locally supported conformal factors $f_n$ such that $R^{g_n}$ is uniformly bounded but $ \vert d {\rm scal}^{g_n} \vert_{g_n} \rightarrow_{n \rightarrow \infty} \infty $, something that cannot happen under $C^3$ convergence of the metric.}, neither assumes any diameter bound nor keeps the topology fixed and concludes $C^k$-subconvergence. Knox (\cite{kK}, Th. 2.1) and Anderson et al. (\cite{AKKLT}, Th. 3.1.1) indicate the possibility to adapt Petersen's \cite{pP} elegant account of Cheeger's result \cite{jC} to manifolds with nonempty boundary to obtain such $C^k$ results. Not doubting that their proof sketches can be made rigorous, the author encountered unexpected technical problems when working on the details. Among others, it is not clear at which point to replace the uniform diameter bound in \cite{pP}. Instead, in this article we follow the strategy to reduce the problem to the one without boundary by an extension operator defined in Theorem \ref{MetricExtension}, embedding manifolds with boundary isometrically into complete closed manifolds while preserving geometric $C^k$ bounds, which could be of independent interest. We build upon results of Schick \cite{Schick} on existence of good atlases and Seeley \cite{rS} on extension of functions on half-spaces.

\subsection{Bounded geometry and main result} 
For a Riemannian metric $h$, we denote by $\Rm_h$ its Riemannian
tensor, and by $\inj_h$ its injectivity radius.  Let
$(M,g,x^0)$ be a pointed connected Riemannian manifold. If $\p M \neq \emptyset$, we denote by $\p
    g=g|_{\p M}$ the induced metric. By $d^g$ we denote the distance
  function induced by the metric $g$, and for a metric space $M$ and $A \subset M$, we write $B(A,r)=\{ \ x \in M \ | \ d(x,A)<r \ \}$.  In the following,
  we adopt the following definition of bounded geometry for manifolds
  with in general nonempty boundary (cf., e.g., \cite{Schick}):
  
  \medskip
  
\begin{definition}\label{def:bounded-geometry}
Fix $k \in \N \cup \{ \infty \}$ and $ c>0$.  A $C^{k+2}$
Riemannian manifold $(M,g)$ has
{\bf $( c,k)$-bounded geometry} if\footnote{From work of Ammann-Gro\ss e-Nistor \cite{AGN} we know that the second condition follows from the other ones.}

\benu
\item[(i)] for the inward normal vector field $\nu$, the normal
  exponential map $E: \p M \times [0, c^{-1}] \rightarrow
  M$, $E(y,r) := \exp_y (r \nu)$, is a diffeomorphism onto its
  image;
\item[(ii)] $\inj_{\p g}(\p M) \ \geq \ 
  c^{-1} $;
\item[(iii)] $\inj_g (M \setminus B (\p M,r)) \geq r $ for all $r \leq
   c^{-1}$;
\item[(iv)] $\vert \nabla_g^l \Rm_g \vert_g
  \leq c$ for all $l \leq k$; 
\item[(v)] $\vert \nabla_{ \p g}^l \SecondFF_g
  \vert_g \ \leq \  c$ for all $l \leq k$, where $\SecondFF$ is the second fundamental form of the boundary.  \eenu
\noindent
For a pointed Riemannian manifold $(M,g,x^0)$ with basepoint $x^0$ we also require 
$d^g(x^0, \partial M) \geq 2 c^{-1}$.
\end{definition}

\begin{remark}
If $(M,g)$ has $(c,k)$-bounded geometry, so has $(\p M, \p g)$,
see \cite{Schick}.\\ In the case $\p M =\emptyset$, Conditions (i),(ii) and (v) are empty, and the condition (iii) reads $\inj_g
 (M)\geq c^{-1}$.
\end{remark}

Our main result is:

\begin{thmA}[Precompactness result]
Let $i \mapsto (M_i, g_i, x^0_i)$ be a sequence of pointed Riemannian
manifolds of dimension $n$ of $( c,k+2)$-bounded
geometry. Then the sequence $\{(M_i, g_i, x^0_i)\}$
$C^k$-subconverges to a pointed manifold $(M_{\infty},
g_{\infty}, x^0_{\infty})$ of $(c,k)$-bounded geometry.\\ If furthermore $\{d_i (x^0_i, \p M_i) \vert i \in \N \}$ is bounded, then $(M_{\infty}, g_{\infty}, x^0_{\infty})$ has non-empty boundary.
\end{thmA}

\begin{remark}
 The condition on the distances $d_i$ cannot be omitted: For the sequence $n \mapsto \{(\overline{B(0,n)},ds^2,0)\}$ of closed balls
  $\overline{B(0,n)}$ around $0$ of radius $n$ in Euclidean $\R^m$, the limit manifold has empty boundary.
  \end{remark}

We review classical results in Sec. 2 and prove the Main Theorem in Sec. 3.

\subsection{Acknowledgments.} The author wishes to thank Boris Botvinnik for initiating this project and for important contributions to this article, Bernd Ammann and Nadine Gro\ss e for insightful discussions, and an anonymous referee for valuable comments on previous versions of the article.

\section{Notation, Convergence for manifolds with boundary}
\subsection{Conventions, notation, definition of convergence}

We first settle some conventions by $0 \in \N$, $\N^* := \N\setminus \{ 0 \}$ and $\N_n := \{ m \in \N \vert m \leq n \} \  \forall n \in \N$ . Let $(Z,d)$ be a metric space, and $Y\subset Z$. For $r>0$ let $B(Y,r)$ be the $d$-ball of radius $r$ around $Y$ in $Z$, and $B(y,r):= B(\{y\},r)$. When it matters to emphasize an ambient space $Z$ or a metric $d$, we use the notation $B^{Z}(y,r)$ resp. $B^d(y,r)$.

If $Z_0,Z_1\subset Z$, then the \emph{Hausdorff distance}
$d(Z_0,Z_1)$ is defined as
\begin{equation*}
d(Z_0,Z_1) = \inf\{ \ r>0 \ | \ Z_0\subset B(Z_1,r)  \land  Z_1\subset
B(Z_0,r) \ \},
\end{equation*}

which generalizes $d$ in the sense that $d(\{ p \}, \{q \}) = d(p,q)$. Balls and distances in a Riemannian manifold $(M,g)$ always refer to the geodesic distance $d^g$. We say (see e.g. \cite{CGT}, \cite{aK}, \cite{jJ}) that a sequence $i \mapsto (M_i,g_i,x^0_i)$ of pointed Riemannian $C^k $ manifolds {\em $C^k$-converges to} a pointed Riemannian $C^k$ manifold
$(M_{\infty},g_{\infty},x_{\infty})$ if and only if for all
$m\in \N \setminus \{ 0 \}$ and for $i$ sufficiently large 
\begin{enumerate}
\item[(a)] there is a an open subset $U_m \supset B (x^0_\infty,m)$ of $M_\infty$ and a diffeomorphism $\phi_i^{(m)} :  U_m \to \phi_i^{(m)} ( U_m) \subset M_i$ with $
  B_m(x_i^0) \subset \phi_i^{(m)} (U_m)$ mapping $x_\infty^0$ to $x_i^0$;
 \item[(b)] the metrics $(\phi^{(m)}_i)^*g_i$ converge to $g_{\infty}\vert_{U_m}$ in
   the $C^k$-norm on $U_m$.
\end{enumerate}

\subsection{Gromov-Hausdorff convergence}

Let $(X,d)$ and $(X',d')$ be metric spaces. A
continuous map $\phi : X\to X'$ is called {\bf $\epsilon$-isometry} if
$||\phi^*d'-d||_{\infty}<\epsilon$.
\begin{definition}\label{def:hausdorff-convergence}
A sequence $i \mapsto (Y_i,d_i,y_i^0)$ of pointed proper complete
metric spaces is said to
{\bf GH-converge} to a complete and proper metric pointed space
$(Y_{\infty}, d_{\infty},y^0_{\infty})$ if one of the following
equivalent conditions is satisfied (see \cite[Section 3.1.2]{Bam}):
\begin{enumerate}
\item[(B$'$)] there are sequences $i \mapsto r_i$, $i \mapsto \epsilon_i$ of
positive real numbers, where $r_i\xrightarrow[i \rightarrow \infty]{} \infty$, $\epsilon_i \xrightarrow[i \to \infty]{} 0$,
and $\epsilon_i$-isometries $\phi_i :
B^{Y_{\infty}}_{r_i}(y^0_{\infty}) \to B^{Y_i}_{r_i}(y^0_{i})$ such that
\begin{equation*}  
B_{\epsilon_i}(\Im \phi_i)\supset B^{Y_i}_{r_i}(y^0_{i}) \ \ \ \mbox{and} \ \ \ d_{i}(\phi_i(y^0_{\infty}),y^0_i)< \epsilon_i.
\end{equation*}
\item[(D$'$)] there is a metric space
$(Z,d)$ and isometric embeddings $\iota_i: Y_i \to Z$,
$\iota_{\infty}: Y_{\infty} \to Z$, such that
\begin{enumerate}
\item[(i)] $\displaystyle \lim_{i\to\infty}\iota_i(y^0_i) =
\iota_{\infty}(y^0_{\infty})$,
\vspace{2mm}

\item[(ii)] $\displaystyle \lim_{i\to\infty} d(U\cap
\iota_i(Y_i), U\cap \iota_{\infty}(Y_{\infty})) =0$ for any open
bounded set $U\subset Z$. 
\end{enumerate}
\end{enumerate}  
We use the notation $\displaystyle {\lim_{i\to\infty}}^{\!\!GH}
(Y_i,d_i,y^0_i) = (Y_{\infty}, d_{\infty},y^0_{\infty})$.
\end{definition}
We need the following particular case of more
general results, see, e.g. \cite[Prop. 3.1.2, Th.
3.1.3]{Bam}:
\begin{theorem}\label{thm:GH-convergence}
Let $A: i \mapsto (X_i,g_i,x^0_i)$ be a sequence of pointed complete
$n$-dimensional Riemannian manifolds such that $\Ric_{g_i}\geq
(n-1)\kappa$ for some $\kappa\in \R$ and all $i \in \N$. Then
there is a pointed proper complete metric space
$(Y_{\infty},d_{\infty}, y^0_{\infty})$ such that the sequence
$A$ GH-subconverges to $(Y_{\infty},
d_{\infty},y^0_{\infty})$.
\end{theorem}  
\subsection{Smooth Cheeger-Gromov convergence}
Let $i \mapsto (X_i,g_i,x^0_i)$ be a sequence of pointed complete Riemannain
manifolds of dimension $n$ which GH-converges to a metric space
$(Y_{\infty},d_{\infty},y_{\infty})$ as in Definition
\ref{def:hausdorff-convergence}.  If the metric space
$(Y_{\infty},d_{\infty},y^0_{\infty})$ is derived from a
Riemannian manifold $(X_{\infty},g_{\infty},x^0_{\infty})$, we use the slightly abusive notation:
$(Y_{\infty},d_{\infty},y^0_{\infty})=(X_{\infty},g_{\infty},x^0_{\infty})$.
\begin{definition}[from \cite{Bam}, Sec. 3.2.1]\label{def:smooth-convegence}
Assume that a sequence $\{(X_i,g_i,x^0_i)\}$ GH-converges to a complete
Riemannian manifold $(X_{\infty},g_{\infty},x^0_{\infty})$.  Then the
sequence {\bf $\{(X_i,g_i,x^0_i)\}$ $C^k$-converges to
$(X_{\infty},g_{\infty},x^0_{\infty})$} if there is an exhaustion of
$X_{\infty}$ by open sets $U_j$, i.e.,
\begin{equation*}
U_1\subset \cdots \subset U_j\subset \cdots \subset X_{\infty},
\ \ \ X_{\infty} = \bigcup_{j}U_j,
\end{equation*}  
and there are diffeomorphisms onto their image $\phi_j: U_j\to M_j$
such that $\phi_j \to Id_{X_{\infty}}$ pointwise, and
\begin{equation*}
\phi^*_j g_j \to_{j \to \infty} g_{\infty} \ \
\end{equation*}
in the pointwise $C^k$ topology, i.e., there is a point-wise\footnote{It is quite obvious from bamler's proof that it is possible to upgrade pointwise convergence in his result to uniform convergence on compacta for any fixed Riemannian metric.} convergence $\phi^*_j g_j \to g_{\infty}$
and $\nabla^{\ell} \phi^*_j g_j \to \nabla^{\ell} g_{\infty}$ for all
$\ell=1,\ldots,k$, where $\nabla$ denotes the Levi-Civita connection
of the metric $g_{\infty}$ on $X_{\infty}$.
\end{definition}
\begin{remark}
Without loss of generality, we assume that $U_j = B(x^0_{\infty}, j) \subset X_{\infty}$ for all $j \in \N \setminus \{ 0 \}$.
\end{remark}
Bamler  (\cite[Theorem 3.2.4]{Bam}) gives a detailed proof of
a Cheeger-type result without a diameter bound:

\begin{theorem}\label{Bamler-thm}{\rm (cf. R. Hamilton \cite{rsH})}
Let $k \in \N \cup \{ \infty \} $, $k \geq 3 $. Let $i \mapsto (X_i,g_i,x_i)$ be a sequence of pointed complete $C^{k+1}$ Riemannian
manifolds of dimension $n$. Assume that $\inj_{g_i}\geq 
c^{-1}$ and $| \nabla^{\ell}\Rm_{g_i} |\leq  c$ for all
$\ell \in \N_{k+1}$ (i.o.w., assume that the sequence is uniformly $(c,k+1)$-bounded). Then the sequence $\{(X_i,g_i,x_i)\}$
$C^k$-subconverges to a pointed complete Riemannian $C^k$ manifold
$(X_{\infty},g_{\infty},x^0_{\infty})$ of dimension $n$ that is $(c,k)$-bounded.
\end{theorem}

\begin{remark}
Strictly speaking, only the case $k = \infty$ is treated in the Theorem 3.2.4 of \cite{Bam}, but its proof contains implicitly the statement for finite $k \geq 3$. The necessity for $k \geq 3$ is in the proof of the induced length structure being the one of a Riemannian manifold. There an Arzel\`a-Ascoli argument is applied to the sequence of metrics $\exp^* g_n$. The loss of one order of differentiability occurs at another instance of the application of the Arzel\`a-Ascoli argument when showing that the limit metric is $C^k$.
\end{remark}

\bigskip

\section{Proof of the main result (Theorem A)}

To apply Theorem \ref{Bamler-thm} we need to make the boundary disappear. The main tool to do so are {\em height functions}. The general strategy of the proof of Theorem A can be summarized as follows:

\begin{enumerate}
\item First we extend in Sec. \ref{extension} the metrics beyond the boundary using height functions; 
\item Then we apply in Sec. \ref{sec:convergence} Hamilton's theorem Th. \ref{Bamler-thm} to the extended metrics;
\item Finally we show in Sec. \ref{restriction} that convergence of the extended metric together with convergence of the height functions entails convergence of the original manifolds-with-boundary.
\end{enumerate}

\subsection{Step 1: An extension procedure}
\label{extension}
\subsubsection{Height functions}
In order to reduce the problem of convergence for manifolds-with-boundary to the corresponding problem for manifolds without boundary, we introduce height functions. To a manifold $M$ with non-empty boundary, we attach a collar to get a complete
manifold $X$ without boundary equipped with a height function $f: X \to (- \infty ; 1]$ such that $M=f^{-1}([0 ; 1])$. Then a sequence $i \mapsto (M_i, g_i, x_i^0)$ of
pointed compact manifolds with non-empty boundary induces a sequence
$i \mapsto (X_i, \hat{g}_i, x_i^0)$ (where $\hat{g}_i$ extends $g_i$) of complete
Riemannian manifolds with height functions.

\begin{definition}\label{def:height}
Let $(X,g,x^0)$ be a pointed Riemannian manifold, $k \in \N$, $c>0$, then $f \in C^k(X , (- \infty ; 1 ])$ is called a $(c, k)$-{\bf height function}, if the
following conditions are satisfied:
\begin{enumerate}
\item[(i)] $\delta^{\p}(f)
  := \min\{ \ |\nabla_{g} f(x)|_{g} \ | \ x\in f^{-1}([ - 1/2 ; 
  1/2] ) \}  \geq c^{-1}$, $f^{-1}( \{ 0 \})\neq \emptyset$,
in particular $0$ is a regular value for the function
$f$;
\item[(ii)] $f(x^0)  >0 $ and $d^g(x^0 ,  f^{-1} (0)) \in (c^{-1}; c)$;
\item[(iii)]  $|\nabla^{\ell}f |\leq c$ for
  all $\ell \in \N_k$.
\end{enumerate}
Of course, if $C \geq c $ and $\kappa \leq k $ then any $(c,k)$- height function is a $(C, \kappa)$-height function. Furthermore, a sequence $\{(M_i, g_i, x^0_i, f_i)\}$ is called {\em of
  $(c,k)$-bounded geometry} if $\{(M_i, g_i, x^0_i)\}$ is a
  sequence of $(c,k)$-bounded geometry and $f_i$ are
  $(c,k)$-height functions on $M_i$.
\end{definition}

\begin{remark}
  Let $(X,g,x^0)$ and $f: X\to (- \infty, 1] $ be as in Definition \ref{def:height}.
  Denote $X^f:=f^{-1}([0 ; 1])$.  Then by
    definition, $X^f$ is a smooth
      manifold with the boundary $\p X^f=f^{-1}(\{0\})\neq \emptyset$,
      i.e., the triple $(X^f,g,x^0)$ is a pointed manifold with non-empty
      boundary. Here we denote by $g$ the restriction $g|_{X^f}$ to
      avoid multiple subscripts in sequences. 
\end{remark}

Now it is easy to see that if $k \geq 1$ and if $f$ is a $(c,k+2)$-height
function on a manifold of $(c,k)$-bounded geometry, then $X^f = f^{-1}
([0,1])$ is a manifold with boundary of $(c,k)$-bounded geometry: The $C^{k+2}$-norm of $f$ controls the $C^k$-norm of its Hessianand thus of the second findamental form of $f^{-1} (0)$. It is a bit
harder to see that also the converse is true:
\begin{theorem}
\label{MetricExtension}
Let $ c>0$, then there exists $\bar c>0$, depending on $ c$ and on $k$ such that,
for any compact pointed manifold $(M,g,x^0)$
of $( c,k)$-bounded geometry  with $\partial M \neq \emptyset$, there exists a pointed isometric inclusion $\iota : (M,g,x^0) \to (X,
\bar{g}, x^0)$ where $(X,\bar g, x^0)$ is a complete pointed manifold
of $(\bar c,k)$-bounded geometry and $(\bar{c},
k)$-height function $f$ on $X$ with $\iota(M)=
f^{-1}([0,1]) $.
\end{theorem}
The proof of this theorem in Subsection \ref{sec:MetricExtension} requires the careful extension of tensors. In a first step, we want to consider an extension procedure for scalar functions instead of tensors.

\subsubsection{Extending functions beyond a boundary}
We will first need a technical result allowing us to extend
functions beyond the boundary of a manifold in a way
that respects infima. To that purpose, let $(M,g,x^0)$ be a
pointed Riemannian manifold  of
$(c,k)$-bounded geometry.

We would like to construct a standard outer collar
to $M$. First, we recall necessary constructions from
\cite{Schick}. Let $(M,g)$ be a Riemannian manifold with non-empty
boundary $\p M$ equipped with the metric $\p g= g|_{\p M}$. We denote
by $\vec \nu$ the inward normal vector field along $\p M$. Then for a
point $x \in \p M$ we fix an orthonormal basis on the tangent space
$T_{x}\p M$ to identify it with $\R^{n-1}$. Then for small enough
$r_1,r_2>0$ there are normal collar coordinates
\begin{equation}\label{eq:kappa}
\kappa_{x}: B(0, r_1)\times [0,r_2) \to M, \ \ \
\kappa_{x}: (v,t) \mapsto \exp^{g}_{\exp_{x}^{\p g}(v)}(t\vec \nu),
\end{equation}
where the exponential maps of $\p M$ and of $M$ are composed. By
assumption, the manifold $(M,g)$ has $(c,k)$-bounded
geometry, in particular, the boundary $(\p M,\p g)$ also has
$(c,k)$-bounded geometry. Let $\delta >0$ s.t. $\p M
\times [0,\delta)$ is
covered by normal collar coordinates charts $U_{\ell}$ with
\begin{equation*}
U_{\ell}= \kappa_{\ell}(V_\ell);
\qquad V_\ell:= W_\ell \times
[0,r_2 ), \qquad W_\ell:= 
B(0, r_1^{(\ell)})\end{equation*}
where $\kappa_{\ell}$ is the corresponding map from
(\ref{eq:kappa}). Since the manifold $(M,g)$ has
$(c,k)$-bounded geometry, \cite[Proposition 3.2]{Schick} implies that
there exist constants $r_0>0$ and $c_0$ and $m_0 \in \N^*$
depending only on $c$ and $k$, such that if $r_1,r_2\leq r_0$ the family of charts $\{\ \kappa_{\ell} \ \vert \ \ell \in
\Z \ \}$ can be chosen locally finite (this finiteness is
controlled by $m_0$), and there is a subordinate
partition of unity $\{ \ \psi_{\ell} \ \vert \ \ell \in
\Z \ \}$ satisfying
\begin{eqnarray}\label{eq:bound-b}
\vert \psi_{\ell}  \vert_{C^k} < c_0 ,
\end{eqnarray}
where, again, $c_0$ only depends on $c$ and $k$, whereas $ r_0, m_0 $ do not depend on $k$ but only on the first three coefficients of bounded geometry (the $C^k$ norm in \cite{Schick} refers to the Euclidean one in charts, and it is shown in Prop. 3.3 ibd. that the chart transition maps are bounded in the $C^k$ norm by a bound defined from the bounded geometry constants). We fix
this atlas $\{\ U_{\ell} \ \vert \ \ell \in \Z \ \}$ of the collar and the subordinate
partition of unity $\{ \ \psi_{\ell} \ \vert \ \ell \in \Z \ \}$  once and forever, as they have the same favorable properties for metrics close
to $g$ as well. Now the atlas $ \bigcup_i \kappa_i \cup \bigcup_j \kappa^{{\rm int}}_j$ (where the $\kappa^{{\rm int}}_j$ are charts for the interior) 
can be extended to an atlas $ \bigcup_i (\hat{\kappa}_i: \hat{U}_i \rightarrow \hat{V}_i) \cup \bigcup_j  \kappa^{{\rm int}}_j$ of a manifold $X$ without boundary diffeomorphic to
the interior of $M$
by extending the original boundary chart images $ V_{i} = W_i \times [0; r_2) $ to $\hat{V}_i:= W_i \times (- \infty ; r_2)$ and the smooth chart transitions from
$$
V_{ij} = W_{ij} \times [0, r_2) \rightarrow V_{ji} = W_{ji} \times [0, r_2)  \ \ \ \mbox{to} \ \ 
W_{ij} \times (- \infty , r_2) \rightarrow W_{ji} \times ( -
\infty , r_2)
$$ (where $V_{ij} := \kappa_i^{-1} (U_i \cap
U_j)$) providing gluing data for a manifold $X$. This manifold preserves the bounds (\ref{eq:bound-b}) for the chart transitions, as the charts considered in \cite[Theorem 2.5]{Schick} are normal boundary charts, whose chart transitions are of the form $(\kappa, \id_{[0;1]}  )$ for a chart $\kappa $ of the boundary. The order\footnote{The {\em order} of a cover is the maximal number of open sets in which a point is contained.} of the new cover $\{ \hat{U}_\ell\}$ is the order of the original cover $\{ U_{\ell} \}$ of the boundary by boundary charts. We call $\{ U_\ell\}$ the {\em cylindrical atlas} and $\{ \hat{U}_\ell \}$ the {\em extended cylindrical atlas}. Each $\psi_i $ constructed in \cite{Schick} with $i \geq 0$, i.e., corresponding to a boundary chart neighborhood, is of the form $\psi_i \circ \kappa_i^{-1} (y',t) = \phi (y') \cdot  \eta (t) $ with $\eta (t)=1$ for all $ t \leq r_2/2$, i.e., the $\psi_i$ do not depend on the neck parameter $t$ near $\partial M$, and we extend them constantly in $t$ to functions $\hat{\psi}_i$, which, consequently, form a partition of unity in $X$ subordinate to the extended cylindrical atlas. 

Let $(X, p)$ be the above extension of the manifold $M$, let $h$ be a complete Riemannian metric on $X$. Let $0<r \leq \infty $, define $B_r := B(x^0,r)
\subset M$. Let $\Lambda_r \subset \Z$ be the subset of indices of boundary charts of the fixed atlas of $M$ whose domains are contained in
$B_r$. Then we define $\p_r M:= \bigcup_{\ell \in \Lambda_r} U_\ell$
and let $X_r := M \cup \bigcup_{i \in \Lambda_r} \hat{U}_i $.
\begin{lemma}
\label{Harnack-extension}{\bf (Stable nonlinear extension operator)}
Let $(X,x^0)$ be the extension of the manifold-with-boundary $M$ as above and let $h$ be a complete Riemannian metric on $X$, which we w.l.o.g. assume to satisfy $\kappa_i^* h> m_0  e_i$ in every chart $\kappa_i$ (where $e_i$ is the Euclidean metric in the chart $\kappa_i$). Then there
is $F: C^0(M, (0; \infty)) \rightarrow C^0(X
, (0; \infty))$ s.t. $F(C^k(M, (0; \infty))) \subset C^k(X, ((0; \infty)))$ and
\begin{enumerate}
\item[{\rm (i)}] the map $F$ is an extension operator, i.e., $  F (u)
\vert_{M} = u $ for all $u \in C^0 (M, (0; \infty))$, and $F_r(u) := F(u) \vert_{X_r} $ only depends on $u \vert_{B_r}$;
\item[{\rm (ii)}] for each $k\geq 1$ and each $b>0$, $F_r$ maps the
space $C^k(B_r , (b; \infty))$ to $C^k( X_r  , \R) $
continuously with respect to the $C^k (B_r)$-norms;
\item[{\rm (iii)}] for each $k\geq 1$, $F_r$ maps $C^k (B_r)$-bounded sets uniformly to
$C^k $-bounded sets, i.e., for every $a>0$ there is a constant $c_1>0$ 
such that
\begin{equation*}
F_r(B^{C^k(B_r)}(0,a) ) \subset B^{C^k (X_r)}_{r\cdot c_1} (0) .
\end{equation*}

\end{enumerate}
Finally, for every $b>0$ there is a constant $\beta \in (0;b)$,
$\beta=\beta(b)$, such that $ u (x) \geq b \forall x \in B_r \Rightarrow (F_r(u)) (x) \geq \beta \forall x \in X $ for all $u$ in a $C^k (B_r)$-neighborhood of metrics. In particular, $F_r$ and $F$ map positive functions to positive functions. 
\end{lemma}
\begin{proof}
Let $\R_+^{n}= \R^{n-1}\times [0,\infty)$
and let $C^{\infty}(\R_+^{n})$ carry the topology of uniform convergence on compact subsets of all
derivatives. In \cite{rS}, Seeley defines a continuous linear
extension operator $E: C^{\infty}(\R_+^{n})\to
C^{\infty}(\R^{n})$, as a sum of dilated reflections. Let $a,b: \N\rightarrow \R$ with $b_k < 0 \ \forall k \in \N$, $b_k \rightarrow_{k \rightarrow \infty} - \infty$ and, for all $n \in \N$: $\sum_{k=0}^{\infty} |a_k| \cdot |b_k|^n < \infty$ and $\sum_{k=0}^{\infty} a_k b_k^n = 1   $, let $\Phi \in C^\infty ([0; \infty], \R)$ with $\supp (\Phi) \subset (- \infty ; 2] $ and $\Phi |_{[0;1]} = 1$. Then for all $t <0$ we define $E(f) (x,t) = \sum_{k=0}^\infty a_k \Phi (b_k \cdot t) f(x,  b_k \cdot  t)$ (and, of course, for $t \geq 0$ we define $(E(f))(x,t) := f(x,t)$). It is obvious from the properties of $\Phi $ that for $B:= \max \{ b_k | k \in \N \} <0 $ we have $\supp (E(f) ) \subset (2/B; \infty ) $.

We define $\Phi \in C^\infty (\R^{n-1}\times [0;r_2), 
\R^{n-1}\times [0;\infty))$ by $\Phi (x, r) := (x, \psi
(r))$ for a stretching diffeomorphism $\psi \in C^\infty ( [0; r_2) , [0; \infty)) $, as well as the operator $E_2:= E
\circ \Phi: C^{\infty}(\R^{n-1}\times [0;r_2)) \to C^\infty
(\R^n)$, and finally an extension operator $E_M: C^\infty( M )
\rightarrow C^\infty(X)$ by 

\begin{equation*}
E_M(u) := \sum \psi_i \cdot (E_2(u \circ \kappa_i) \circ \hat{\kappa}_i^{-1}) . 
\end{equation*}
It is well-defined for the same reason as above, and by inspection, it is clear from \cite{rS} that $E_M$ satisfies the
above properties (i), (ii) and (iii) (here we use that the respective metrics in every boundary chart of $M$
satisfy $C^k$-bounds with respect to the Euclidean
metric on open subsets of half-spaces.).

To construct an extension respecting lower bounds, we define:
\begin{equation}
	\label{Positive-Extension}
F (u)  (x):= \exp (E_M (\ln u (x))) .
\end{equation}
The properties (i)--(iii) are transferred from $E_M$ to
$F$ by uniform continuity of
$\ln\vert_{[\sigma; \infty)}$ for any $\sigma>0$. As
$\ln\vert_{[\sigma; \infty)}$ is bounded away from $-\infty$, there
is $\beta=\beta(b)$ with $ \inf (F(u)
\vert_{X_r} )\geq \beta $ if $ \inf (u
\vert_{ B_r} ) \geq b$.
\end{proof}

\subsubsection{Proof of Theorem \ref{MetricExtension}}
\label{sec:MetricExtension} We extend the proof of Lemma \ref{Harnack-extension} to endomorphism-valued functions: Let $\kappa_\ell : V_\ell \rightarrow U_\ell $ be a member of the cylindrical atlas, denote by $g_{\ell}$ the metric $g$ restricted to $U_{\ell}$, and
by $e_{\ell}=(\kappa_{\ell}^{-1})^*(ds^2)$, where $ds^2$ is the
Euclidean metric on $B(0, r_1^{(\ell)})\times [0;r_2)$.
For each $\ell$, we define
\begin{eqnarray}
A_{\ell} := e^{-1}_{\ell} \circ g_{\ell}  : TU_{\ell} \rightarrow TU_{\ell}
\end{eqnarray}
(the metrics understood as maps $TU_{\ell} \rightarrow T^*
U_{\ell}$). The $A_{\ell}$ are positive-definite symmetric
operators, so their spectrum is contained in $(0; \infty)$. By
\cite[Proposition 2.3]{Schick}, there is
$a_0 \in (0; \infty)$ with $\vert \vert A_{\ell}  (p) \vert \vert, || A^{-1}_{\ell} (p) || \leq a_0 $ for all $p$. This allows to put ${\mathbf a}_{\ell}:= \ln (A_{\ell})$,
which are smooth maps from $U_\ell$ to the set of symmetric matrices
${\rm Mat}_s(\R^n, \R^n)$ whose operator norm is bounded by $\ln a_0$. By $F$ from
Lemma \ref{Harnack-extension}, Eq. \ref{Positive-Extension}, extend the coefficients of each
matrix ${\mathbf a}_{\ell}$ to the members of the extended atlas
\begin{equation*}
\hat{U}_{\ell} :=  \kappa_\ell (V_\ell) , \qquad V_\ell:= B(0, r_1^{(\ell)})\times
(- \infty ; r_2).
\end{equation*}
This gives maps $\hat{\mathbf a}_{\ell} : \hat{U}_\ell \rightarrow
{\rm Mat}_s(\R^n, \R^n) $, such that $\hat{\mathbf
	a}_{\ell}|_{U_{\ell}}={\mathbf a}_{\ell}$.

For elements of the partition of unity $\{ \psi_k| k \in \Z\}$ fixed above, we define their extensions $\hat{\psi}_{\ell} $ not by the Seeley extensions (as they would not preserve the partition of unity property) but instead using that, close to the boundary, each function in the partition of unity does not depend on the neck parameter $t$, as seen above. Consequently, we can, and will, define $\hat{\psi}_{\ell} $ by extending them constantly in $t$-direction. Then we define $ \hat{A}_{\ell} :=
\exp(\hat{\mathbf a}_{\ell}) $, which is a positive-definite symmetric smooth extension of $A_{\ell} $, and we define the Riemannian metric $\hat{g}_{\ell}:= e_{\ell}\circ \hat{A}_{\ell}$ on $ \hat{U}_{\ell} $. Finally, we
put $ \hat g:= \sum \hat{\psi}_{\ell} \cdot \hat{g}_{\ell}$, which is a metric on the manifold $X$ which we now show to be complete. Let
$e_{\ell}$ be the Euclidean metric in $\hat U_{\ell}$,
and, for $r_1$ and $r_2$ as in the definition of bounded geometry define the metric $\bar g_{\ell}:=\hat \kappa_{\ell}^* \hat g$ on
\begin{equation*}
B(0, r_1^{(\ell)})\times (- \infty ;r_2) \subset \R^n.
\end{equation*}
Let $\bar A_{\ell}:= \bar g_{\ell} \circ e_\ell^{-1}$, then the operator norms of the
$\bar A_{\ell}$ are bounded above. But also the norm of the inverses is bounded, as by the definition of bounded geometry, for 
each point $x\in \hat M$, there are at most $m_0$ neighborhoods
$\hat{U}_{\ell}$ with $x\in \hat{U}_{\ell}$, thus there is an index
${\ell}_0$ such that $\psi_{\ell_0}(x) \geq m_0^{-1}$, and for all $v \in TM$ we get
\begin{eqnarray}
e( \bar{A}_{\ell} v , v) = \sum_{\ell'} \hat \psi_{\ell'}
\hat{A}_{\ell'} (v,v) \geq m_0^{-1} \hat{A}_{\ell_0} (v,v)
\end{eqnarray}
for some $\ell_0$, as all summands are nonnegative (we had seen above that the $\hat{A}_l$, being exponentials of symmetric matrices, are positive-definite). Now $\hat{A}_{\ell_0}
(v,v)$ in turn can be estimated by
\begin{eqnarray}
\hat{A}_{\ell_0} (v,v) \geq \beta(a_0^{-1} ) \cdot  \vert  v \vert  ,
\end{eqnarray}
which implies completeness via the Cauchy criterion. These and the more obvious\footnote{Uniform $C^k$ bounds to the $A_\ell$ imply uniform $C^k$ bounds to the $a_\ell = \l (A_\ell)$ which, via the Seeley extension, imply uniform $C^k$ bounds to the $\hat{a}_\ell$ and thus uniform $C^k$ bounds to the $\hat{A}_\ell := \exp (\hat{a}_\ell)$, so by the bound on the order of the extended cylindrical atlas and the uniform $C^k$ bounds on the $\psi_\ell$ imply uniform $C^k$ bounds on the $g_\ell$ and thus on $\overline{g}_\ell$.} $C^k$-bounds from $\infty$ are exactly the estimates needed to show $(c,k)$-bounded geometry of $(X, \hat{g}) $. As a height function we take, for $\tau \in C^\infty ((- \infty; r_2] ) $ with $\tau (r) = r $ for all $r \in (- \infty; r_2/4]$ and $\tau([r_2/2; r_2] ) = r_2/2$, 
\begin{eqnarray}
f:= \sum_{\ell} \psi_{\ell} \cdot  (\tau \circ {\rm pr}_2 \circ \kappa_{\ell}),
\end{eqnarray}
complemented by $r_2/2$ in the interior, which after a final rescaling (to have the required range $(- \infty; 1]$ and satisfy the bound from zero on $f^{-1}([- 1/2;1/2] )$) satisfies all requirements: $\pr_2 \o \kappa_{\ell}$ is one of the coordinates, called $x_0$ in \cite{Schick}, whose $C^k$ norm is uniformly bounded by a constant $\overline{c}$ called $D$ in \cite{Schick} (Th. 2.5.(b1)), that depends only on the constant $C$ as in $(C,k)$-bounded. \hfill \qed

\subsection{Step 2: Convergence with height functions}
\label{sec:convergence}

\begin{theorem}\label{thm:convergence}
Let $i \mapsto (X_i,g_i,x^0_i,f_i)$ be a sequence of complete pointed $C^k$
manifolds equipped with height
functions of $(c,k+1)$-bounded geometry with $ c>0$, $k\geq 4$.  Then the sequence $i \mapsto (X_i,g_i,x^0_i,f_i)$
 $C^k$-subconverges\footnote{We can even conclude $C^{k, \alpha}$ subconvergence by the used Arzel\`a-Ascoli arguments but let us focus on $C^k$ spaces.} to $(X_{\infty},g_{\infty},x^0_{\infty},f_{\infty})$,
\ where $(X_{\infty},g_{\infty},x^0_{\infty}) $ \ is a complete open
manifold, and $f_{\infty}: X_{\infty}\to (-\infty, 1]$ is a
$(c,k )$-height function.
\end{theorem}

\noindent {\em Proof}. Let $i \mapsto (X_i,g_i,x^0_i,f_i)$ be a sequence from Theorem
\ref{thm:convergence}. By Theorem \ref{Bamler-thm} we may assume that the sequence of
manifolds $i \mapsto (X_i,g_i,x^0_i)$ already
$C^k$-converges to a pointed complete Riemannian manifold
$(X_{\infty},g_{\infty},x^0_{\infty})$ of dimension $n$. Without loss of
generality, we can also assume that the exhaustions of $X_{\infty}$ is
chosen as a systems of open balls $\{B(x^0_{\infty}, i) \vert i \in \N \}$.  Let $\phi_i
:B(x_{\infty}, i)\to X_i$ be the diffeomorphisms (on their image) from
Definition \ref{def:smooth-convegence}.

We recall that $f_i : X_i \to \R$ are
$(c,k+1)$-height functions as in Definition
\ref{def:height}.  By definition, we have $|\nabla^{\ell}f_i|\leq
    c$ for all $\ell=0,1,\ldots, k+1$.  Then by passing to
    a subsequence if necessary, the functions $\tilde{f}_i:=
    \phi_i^* f_i$ also $C^k$-converge to a $C^k$ function $f_{\infty}:
    X_{\infty}\to \R$ on the ball $B(x^0,n) $, and then we choose a diagonal sequence to get convergence on every ball. By assumptions on the sequence of functions
    $i \mapsto f_i $, it is evident that the function $f_{\infty}$ is also a
    $(c,k)$-height function, and since $\delta^{\p}(f_i)\geq c^{-1}$
    for all $i=1,2,\ldots$ (where $\delta^\partial $ is defined as in Definition \ref{def:height}), we obtain that
    $\delta^{\p}(f_{\infty})\geq c^{-1}$.  This completes the proof of
    Theorem \ref{thm:convergence}. \hfill \qed

\subsection{Step 3: Coming back to manifolds-with-boundary}
\label{restriction}

\begin{theorem}
\label{cutting}
Let $(X_i, g_i, x_i)$ be a $C^k$-convergent sequence and let $f_i: X_i \rightarrow \R $ be $(c,k)$-height functions, then the sequence $ i \mapsto (M_i := X_i^{f_i}, g_{M_i}, x^0_i)$ of manifolds with boundary $C^{k-1}$-subconverges to $(M_\infty := X_\infty^{f_\infty}, g_{M_\infty}, x^0_\infty) $.
\end{theorem}

\begin{proof} The manifolds $(M_i := X_i^{f_i},
g_i , x^0_i)$ Hausdorff-converge (as closures of open subsets of the limit manifold) and thus Gromov-Hausdorff-converge to
$(M_{\infty} := X_{\infty}^{f_\infty}, g_{\infty}, x^0_{\infty})$ as pointed metric spaces. However, we need more: we need to prove Cheeger-Gromov convergence.
Now, for any fixed radius $r$ and $i>r$, we construct diffeomorphisms $D_i^r$ from the manifold with height function $ H:= B (x^0_\infty, r) \cap f_\infty^{-1} (( - 1/2,  1/2)) $ to an open set in $B (x^0_\infty,r) \cap \tilde{f}_i^{-1} (( - 1/2,  1/2)) $ (recall that $x^0_\infty = \Phi_i^{-1} (x^0_i)$ and that the $\tilde{f}_i$ are defined as in Section \ref{sec:convergence}) by means of the gradient flows of $\tilde{f}_i$:
\begin{equation*}
D_{i, \p}^r (x) := {\rm Fl}_{{\rm grad} \tilde{f}_i}^{t(x)} (x) \ \forall x \in H , 
\end{equation*}
\V where $t(x)$ is chosen such that 
\begin{equation*}
\tilde{f}_i( {\rm Fl}_{{\rm grad} \tilde{f}_i}^{t(x)} (x))  = f_\infty(x).
\end{equation*}
In general, $t(x)$ can depend on $i$. It is easy to see that $t$ is a smooth function (using the product decomposition of a neighborhood $U$ of $\tilde{f}_i^{-1} (0)$ given by the gradient flow of $\tilde{f}_i$ and the fact that $B (x^0_\infty, r) \cap f_\infty^{-1}((- 1/2, 1/2)) \subset U$), and $D_{i, \p}^r$ is a diffeomorphism onto its image. Standard integral estimates yield
\begin{equation*}
\vert  \tilde{f}_i (  {\rm Fl}_{{\rm grad} \tilde{f}_i}^{t(x)} (x) ) - \tilde{f}_i (x) \vert \geq t(x) \cdot \delta^\p (\tilde{f}_i )  \ \forall x \in H,
\end{equation*}
($\delta^\partial$ defined as in Definition \ref{def:height}) and with this uniform flow time estimate and the $C^k$-estimates on
${\rm grad} \tilde{f}_i$, we get $C^k$-bounds of the $D_{i, \p}^r$
tending to $0$. Now $D_{i, \p}^r$ is a diffeomorphism from $H$ to its
image, which is in $ B(x^0_\infty, r+1) \cap X_\infty^{f_\infty} $. For $i$ large
enough, we get $d(D_{i, \p}^r y, y) $ smaller than the convexity
radius in $B(x^0_i, r+1)$. This allows us to interpolate $D^r_{i, \p} $
geodesically with the identity in ${\rm int} (M)$: We define
\begin{equation*}
D_i^r (y) := \exp_y ( \phi (y) \cdot \exp_y^{-1} (D_{i, \p}^r (y)) )
\end{equation*}
for $y \in B_r (x_\infty) \cap \tilde{f}_\infty^{-1} ( - 1/2, 1/2)$
and for a smooth function $\phi$ supported in $B (x^0_\infty,r) \cap
\tilde{f}_\infty^{-1} ( - 1/2, 1/2)$ and identical to $1$ in a
neighborhood of $\tilde{f}_{\infty}^{-1} (0)$, and extended by $D_i^r
(y) =y$ on the complement, getting a sequence of diffeomorphisms from
$M_i \cap B(x^0_i, r)$ as above that $C^k$-converges as well. The image of $D_i^r$
is still contained in $ B(x^0_i, r+1) \cap X_\infty^{f_\infty} $ and
contains $B_{r-1} \cap X_\infty^{f_\infty}$, which allows to show $M_i
\rightarrow_{C^k} M_\infty$.
\end{proof}

\subsubsection{Proof of Theorem A}

Now let us prove Theorem A. Assume that we are given a sequence of
pointed manifold with boundary $(M_i,g_i, x^0_i)$ of $( c,k+2)$-bounded geometry. Then we can extend every $(M_i,g_i, x^0_i)$
to a pointed boundaryless manifold $(X_i,\overline{g}_i, x^0_i, f_i)$ as
in Theorem \ref{MetricExtension}, yielding a sequence of $(c_2, k+2)$-bounded geometry, for some $c_2>0$. Then Theorem \ref{thm:convergence}
implies that there is a $C^{k+1}$-convergent subsequence for both manifolds and
height functions, also denoted by $(X_i,\overline{g}_i, x^0_i,
f_i)$. Finally, Proposition \ref{cutting} implies that, in the $C^k$ sense,

\begin{equation*}
\lim_{i \rightarrow \infty} (M_i, g_i, x^0_i) = \lim_{i \rightarrow \infty} (X_i^{f_i}, \overline{g}_i, x^0_i) = (X_\infty^{f_\infty} , \overline{g}_\infty, x^0_i ) .
\end{equation*}

\V The last assertion follows from the existence of a zero locus of the height function, which in turn follows from the fact that the base point is mapped to a positive value and that there are also negative values in the image of the limiting function due to the definition of convergence. \hfill \qed


\end{document}